\documentclass[12pt]{amsart}
\usepackage{amssymb,lscape}

  \theoremstyle{definition}
  \newtheorem{definition}{Definition}[section]
  \newtheorem{example}[definition]{Example}

  \theoremstyle{plain}
  \newtheorem{lemma}[definition]{Lemma}
  \newtheorem{proposition}[definition]{Proposition}
  \newtheorem{theorem}[definition]{Theorem}

 \newcommand{\n}{\nonumber}

\begin{document}

\title{Some Residually Solvable One-Relator Groups}

\author{Delaram Kahrobaei}

\address{Doctoral Program in Computer Science Department, CUNY Graduate Center, City University of New York\\365 Fifth Ave., USA 10016}

\email{dkahrobaei@gc.cuny.edu}

\author{Andrew F. Douglas}

\address{Department of Mathematics, New York City College of Technology,
CUNY, 300 Jay Street, Brooklyn, New York, USA 11201}

\email{adouglas@citytech.cuny.edu}

\author{Katalin Bencs\'{a}th}

\address{Department of Mathematics and Computer Science, Manhattan College, Riverdale, New York 10471}

\email{katalin.bencsath@manhattan.edu}

\thanks{Partially supported by FFPP/CUNY, and PSC/CUNY}
\begin{abstract}

This communication %is concerned with
%In this paper we
records some observations made in the course of
%the
studying %of
one-relator groups from the point of view of residual solvability.
As a contribution to classification efforts we single out %describe
some relator types that render %make a
the corresponding one-relator groups residually solvable.
\end{abstract}
\maketitle

\section{Introduction}
It is well known that free groups are residually nilpotent and,
consequently, residually solvable. There is a sizable amount of
literature devoted to investigations when (residual, virtual)
properties of free groups are inherited by one-relator groups. The
purpose of this communication is to offer a collection of facts and
examples gathered in the course of attempts to characterize the
residually solvable one-relator groups in terms of the (form of) the
single defining relator. We prove a number of sufficiency results
for certain cases when the relator is a commutator, and then raise
some questions. We start with reviewing some of the literature that
motivated our interest in the topic.

G. Baumslag in \cite{baum1} showed that positive one-relator groups,
which is to say that the relator has only positive exponents, are
residually solvable. In the same paper he provides a specific
example to show that not all one-relator groups are residually
solvable. A free-by-cyclic group is necessarily residually solvable.
It is worth mentioning that the Baumslag-Solitar groups $B_{1,n}$
are solvable, but not polycyclic. On the other hand $B_{m,n}$ are
free-by-solvable by a result of Peter Kropholler \cite{PK90} who
showed their second derived subgroup to be free, therefore
residually solvable. However, there are non-Hopfian
Baumslag-Solitar's groups amongst them, and those are certainly not
residually nilpotent. In \cite{BFMT} Baumslag, Fine, Miller and
Troeger established that many one-relator groups, in particular
cyclically pinched one-relator groups are either free-by-cyclic or
virtually free-by-cyclic. These point to another large class of
residually solvable one-relator groups. In view of the recent result
of M. Sapir and I.Spakulova in \cite{SS} that with probability
tending to $1$, a one-relator group with at least $3$ generators and
the relator of length $n$ is residually finite, even virtually
residually (finite $p$)-group and coherent for all sufficiently
large $p$. In addition in \cite{S10} M. Sapir focuses on residual
properties of one-relator groups with at least $3$ generators.\\

Accordingly our focus is mainly on two generator one-relator groups.
Our two main results concern the situation where $G$ is a
one-relator group whose relator is a commutator.  We provide certain
sufficient conditions where $G$ is residually solvable but also give
examples to show that in general almost anything can occur.

Clearly the attempts to find criteria for residual solvability are
influenced by outcomes of recent and older studies on (fully)
residually freeness of one-relator groups, in particular. Examples
of such one-relator groups are surface groups, which are known to be
fully residually free \cite{CG05} and therefore are residually
solvable. Also in \cite{bb67}, B.Baumslag shows residual freeness of
the one-relator groups of the type $\langle a_1, \cdots, a_k;
{a_1}^{w_1} \cdots {a_k}^{w_k} =1 \rangle$ where $k >3$, and $w_i$s
in the ambient free group on $a_1, \cdots, a_k$ satisfy certain
conditions, implies residual solvability of such one-relator groups.

%The organization of the rest of the paper is as follows:
%Section $2$ reviews background material of various kinds; section $3$ recounts some of the results in **[2] and %shows how they are applicable to those one-relator groups in which the relator is the commutator of positive %words. Said observation also motivates the choice of  the non-positive one-relator groups examined subsequently %%in Section $4$.
%Section **$5$ establishes the residual solvability of the group $G=\langle a,b;[a,w] \rangle$ in the case when %$w=[a,b]^n$ through exploiting the connections between generalized free products and one-relator groups in a %careful adaptation of the Magnus breakdown method. The relator of the group studied in section $6$ is a basic %commutator of a certain type, and residual solvability of the group is the by-product of its free-by-cyclic %structure.
%In section $7$ we list some problems of interest we believe to be open.

\section{Preliminaries}
For convenience, we start with a list of some the definitions,
facts, and theorems that form our starting base.

\begin{theorem}\label{von} (Von Dyck)
Suppose $G =\langle X; R \rangle$ and $D =\langle X; R  \cup S
\rangle$, with presentation maps $\gamma$  and $\mu$ respectively.
Then $x \mu \mapsto x \gamma$  $(x \in X)$ defines a homomorphism of
$G$ onto $D.$
\end{theorem}

\begin{theorem}\label{frei} (Freiheitssatz) \cite{mag}
Let $G$ be a one-relator group, i.e., $G = \langle  x_1, \cdots,x_q;
r = 1 \rangle$. Suppose that the relator $r$ is cyclically reduced,
i.e., the first and the last letters in r are not (formal) inverses
of each other. If each of $x_1,\cdots,x_q$ actually appears in $r$,
then any proper subset of $\{x_1,\cdots,x_q\}$ is a free basis for a
free subgroup of $G$.
\end{theorem}

W. Magnus' method of structure analysis \cite{mag} for groups with a
single defining relation  has the following immediate consequence:

\begin{lemma} \label{mag1} Let  $G= \langle b, x,\cdots,c;r=1 \rangle$ be a one-relator group. Suppose that $b$ occurs in $r$ with
exponent sum zero and that upon re-expressing $r$ in terms of the
conjugates $b^i x b^{-i} = x_i, \cdots, b^i c b^{-i} = c_i$, $(i \in
\mathbb{Z})$ and renaming $r$ as $r_0$, $\mu$  and  $\nu$ are
respectively the minimum and maximum subscripts of $x$ occurring in
$r_0$. If   $\mu < \nu$ and if both $x_{\mu}$  and  $\nu$  occur
only once in $r_0$ then $N= gp_G (x,\cdots,c)$  is free. If $G$ is a
two-generator group with generators $b$ and $x$, then $N$ is free of
rank  $\nu - \mu+1$.
\end{lemma}

\begin{definition} \label{resid} A group $G$ is \emph{residually solvable} if for each $w \in G$ $(w \neq 1)$, there exists
a solvable group $S_w$ and an epimorphism $\phi : G \longrightarrow
S_w$ such that $w \phi \neq 1$.
\end{definition}

\begin{theorem} (Kahrobaei \cite{DK04}, \cite{DK10})\label{az2}
Any generalized free product of two finitely generated torsion-free
nilpotent groups, amalgamating a cyclic subgroup is an extension of
a residually solvable group by a solvable group, therefore is
residually solvable.
\end{theorem}

\begin{theorem} (Kahrobaei \cite{DK04}, \cite{DK10})\label{az3}
Any generalized free product of an arbitrary number of finitely
generated nilpotent groups of bounded class, amalgamating a subgroup
central in each of the factors, is an extension of a free group by a
nilpotent group, therefore is residually solvable.
 \end{theorem}

\begin{theorem} (Kahrobaei \cite{DK04}, \cite{DK10}) \label{az4}
The generalized free product of a finitely generated torsion-free
abelian group and a nilpotent group is (residually
solvable)-by-abelian-by-(finite abelian), consequently residually
solvable.
\end{theorem}

Note that the groups in all three of these theorems above satisfy
the conditions of K. Gruenberg's portent observation \cite{gruenb}
that we record here as
\begin{lemma} \label{P}
Suppose $P$ is any group, $K  \lhd P$ with $P/K$ solvable and $K$
residually solvable. Then $P$ is residually solvable.
\end{lemma}

%So we immediately have:
%\begin{corollary} \label{rs}
%The groups of \ref{az2}, \ref{az3}, \ref{az4} are all residually
%solvable.
%\end{corollary}
\section{The single relator is a commutator}
We first recall a result from \cite{baum2} for a particular class of
non-positive one-relator groups. Let $G$ be a group that can be
presented in the form,
\begin{equation}\label{first}
G= \langle t,a,...,c;uw^{-1} =1 \rangle,
\end{equation}
where $u$ and $w$ are positive words in the given generators and
each generator occurs with exponent sum zero in $uw^{-1}$. %Here a
%word is termed positive if only non-negative powers of  $t,a,...,c$
%occur in the relator.
Then $G$ is residually solvable. In fact, $G$ is free-by-cyclic.

Now consider the group,
\begin{equation}
H = \langle t, a,..., c; [u,w]=1  \rangle.
\end{equation}
If $u$ and $w$ are positive, $H$ can be recognized as one of the
groups in the preceding class \eqref{first}. Hence $H$ is free-by-cyclic and therefore
residually solvable.
However, known examples show that residual solvability for $H$ is not guaranteed once the requirement that
both $u$ or $w$ be positive is relaxed: %for example,   in [2]  in
%general  : G.Baumslag gave examples of $H$ that are
%not residually solvable (for instance, see example \ref{ex1} below).

%\section{ Other non-positive one-relator groups:  Some examples and comments}

%We now present examples of general one-relator groups in which the
%relator is non-positive.   Some will be residually solvable, others
%will not.
%\begin{example}
%Let $G = \langle a,b; uw^{-1}=1  \rangle$  where $u$ and $w$ are
%positive, and that both generators appear in $uw^{-1}$. If $\alpha$
%is the exponent sum of $a$ in $uw^{-1}$,  and if $\beta$  is that of
%$b$, with $\alpha  > 0 > \beta$ then $G$ is free-by-cyclic, [2].
%Therefore $G$ is residually solvable.
%\end{example}

\begin{example} \label{ex1} \cite{baum2}
If $G = \langle  a,b,...,c; [u,v]=1  \rangle$ where
\begin{equation}
u=a, ~ v=[a,b][w,w^b], ~\text{and}~ [a,b]^{-1} [a,b]^a, \n
\end{equation}
then $G$ is not residually solvable. For it follows from Magnus'
solution of the word problem that $w \neq 1$ \cite{mag}. Furthermore
since $[u,v]=1$ we find that
\begin{equation}
[a,b]^a[w,w^b]^a=[a,b][w,w^b],\n
\end{equation}
so that
\begin{equation}
w=[a,b]^{-1}[a,b]^a=[w,w^b]([w,w^b]^a)^{-1}.\n
\end{equation}
Thus $w$ lies in every term of the derived series of $G$.
\end{example}

In contrast, the next example is a residually solvable one-relator group.

\begin{example}
The group $G=\langle a,b; [a,[a,b]] \rangle$  is free-by-cyclic.
\end{example}
\begin{proof}
We expand and re-express the relator,
\begin{equation}
r=[a,[a,b]]=a^{-1}[a,b]^{-1}a[a,b]=a^{-1}b^{-1}a^{-1}bab^{-1}ab.
\end{equation}
Observe that in $$r_0 =b_1^{-1} b_2 b_1^{-1} b_0$$ and $\mu =2$,
$\nu = 0$, and $b_0$ and $b_2$ both occur only once and we can
invoke lemma \ref{mag1}. Therefore $G$, as a cyclic extension of the
free group $N=gp_G (b)$ is residually solvable by (c.f. \ref{P}).
\end{proof}

\section{Connection between Generalized Free Products and
One-Relator Groups}
Over the years since W. Magnus developed his
treatment of one-relator groups the increased interest in them
yielded many new results. Karrass-Solitar in $1971$ showed that a
subgroup of a one-relator group is either solvable or contains a
free subgroup of rank two. Baumslag-Shalen showed that every
one-relator group with at least four generators can be decomposed
into a generalized free product of two groups where the amalgamated
subgroup is proper in one factor and of infinite index in the other.
Fine-Howie-Rosenberger \cite{fine} and Culler-Morgan \cite{CM87}
showed that any one-relator group with torsion and at least three
generators can be decomposed, in a non-trivial way, as an
amalgamated free product.

These results made it seem a reasonable preassumption that a closer
look at the residual solvability of generalized free products of two
groups could help in establishing residual solvability of further
one-relator groups. The following result confirms that assumption.

\begin{theorem} \label{main1}
The group $G=\langle a,b; [a,w] \rangle$, where $w=[a,b]^n$ ($n \in
\mathbb{N}$), is residually solvable.
\end{theorem}

\begin{proof}
Put $N=gp_G(b)$, the normal closure of $b$ in $G$. Using the Magnus
break-down, we consider:
\begin{equation}
N_0 = \langle b_0,b_1,b_2; (b_1 b_0)^n = (b_2 ^{-1} b_1)^n \rangle.
\end{equation}
Now let $$x_0 =b_1^{-1} b_0, ~ x_1 = b_2^{-1} b_1,~ y=b_1.$$ Tietze
transformations confirm that
\begin{equation}
N_0 = \langle x_0 x_1 y; (x_0)^n =(x_1)^n  \rangle = \langle x_0, x_1;  (x_0)^n =(x_1)^n \rangle \ast \langle y \rangle.
\end{equation}
Next let $K= \langle x_0, x_1; (x_0)^n = (x_1)^n \rangle$. Clearly
\begin{equation}
K =  \{ \langle x_0 \rangle \ast \langle x_1 \rangle;  \langle x_0^n \rangle = \langle x_1^n \rangle \}
\end{equation}
Since each factor of $K$ is abelian, by theorem \ref{az2} $K$ is
residually solvable. The free factor of $N_0$, $\langle y \rangle$
is also residually solvable. Therefore $N_0$  is residually
solvable, and it follows for every $i \in \mathbb{N}$ that  $N_i$
is residually solvable. If we put $N_{i,j}=gp(N_i,N_{i+1},...,N_{j})$, the proceeding approach gives

\begin{equation}
N_{i,j}= \langle x_i \rangle \ast_{\langle   x_i^n \rangle = \langle x_{i+1}^n \rangle} \ast \langle  x_{i+1} \rangle
\ast \cdot \cdot \cdot \ast \langle x_j \rangle \ast _{ \langle    x_{j}^n  \rangle = \langle  x_{j+1}^n \rangle    } \ast \langle x_{j+1} \rangle \ast \langle y \rangle.\n
\end{equation}

Therefore, by Theorem \ref{az3}, $N_{i,j}$ is residually solvable. A task that remains for completing the proof is to show that the ascending union $N= \cup_{r<0; s>0  } N_{r,s}$
is residually solvable, which will be taken care of by the following proposition \ref{below}. Granted that, the residual solvability of $G$ follows with the use of
corollary \ref{P}.
\end{proof}

\begin{proposition}\label{below}
$N= \cup_{r<0; s>0}  N_{r,s}$ is residually solvable.
\end{proposition}
\begin{proof}
We will retain notation from the proof of Theorem \ref{main1} and
start with the assumption that $N_{i,j}$ is residually solvable for
all $\ i,j \in \mathbb{N}$ ($(i \leq j)$. For the derived series of
$N= \cup_{r<0; s>0  } N_{r,s}$, we have
\begin{equation}
\delta_i N = \delta_i (\cup_{r<0; s>0  }  N_{r,s}). \n
\end{equation}
Every element in  $g \in \delta_i N$ is a finite product of
commutators of elements from a (finite) subset of the  $N_{r,s}$
groups. So $g \in \delta_i N_{r,s}$  for suitably small  value of
$r<0$ and suitably large value of $s>0$. Thus  $\delta_i N=
\cup_{r<0; s>0  }  N_{r,s}$. Now, if $j$, $k$ are a pair of fixed
integers the infinite union above can be rewritten as
\begin{equation}
\delta_i N =  \cup_{r<0; s>0  } \delta_i N_{j+r,k+s},\n
\end{equation}
so that,
\begin{equation}
\delta_iN \cap N_{j,k}=  (\cup_{r<0; s>0  } \delta_i N_{j+r,k+s}) ) \cap N_{j,k}.\n
\end{equation}
Equivalently,
\begin{equation}
\delta_i N \cap N_{j,k}=  \cup_{r<0; s>0  } (\delta_i N_{j+r,k+s} \cap  N_{j,k}).\n
\end{equation}
Further, each term in the union an be written as
\begin{equation}
\delta_i  N_{j+r,k+s} \cap N_{j,k}= (\delta_i N_{j,k+s} \cap N_{j,k}) \cap (\delta_i N_{j+r,k} \cap  N_{j,k}).\n
\end{equation}
And because $s < 0$ and $r > 0$, an argument fashioned after that in \cite{baum1} (p. $175$, Lemma $4.3$) yields, that after conjugations by suitable powers %of $d$ that
  \begin{equation}
\delta_i N_{j,k+s} \cap N_{j,k}= \delta_i N_{j,k},~ \text{and} \n
\end{equation}
\begin{equation}
\delta_i N_{j+r,k} \cap N_{j,k}= \delta_i N_{j,k}.\n
\end{equation}
So each term in the union can be re-expressed as
\begin{equation}
\delta_i N_{j+r,k+s} \cap N_{j,k}= \delta_i N_{j,k}.\n
\end{equation}
Notice that this expression is independent of $r$ and $s$. Thus, we get
\begin{equation}
\delta_i N  \cap N_{j,k}= \delta_i N_{j,k}
\text{(\cite{baum1}[\text{pg}. 175, \text{line} 16])}.\n
\end{equation}
We claim that
\begin{equation}
\delta_i N  \cap N_{j,k}= \delta_i N_{j,k}.\n
\end{equation}
implies that $N$ is residually solvable.
To see this, let  $a$  be a non-trivial element of
\begin{equation}
N= \cup_{r<0; s>0}  N_{r,s}. \n
\end{equation}
Then, there is an integer  $j=j(a) \in \mathbb{N}$ such that $a \in N_{-j,j}$. By our (inductive) hypothesis at the outset % $\forall i,j \in \mathbb{N}$, $N_{i,j}$
%is residually solvable, so we know that
$N_{-j,j}$  is residually solvable. Consequently,there exists an integer $i \in \mathbb{N}$  such that $a \notin \delta_i N_{-j,j}$.  Then, since  $\delta_i N \cup N_{-j,j} =\delta_i N_{-j,j}$, we see that $a \notin \delta_i N \cap N_{-j,j}$.  But $a \in N_{-j,j}$. So it must be the case that $a \notin \delta_i N$. Thus we have found a normal subgroup $\delta_i N \lhd N$  with the property that  $a \notin \delta_i N$ and   $N/ \delta_i N$ is solvable. Hence $N$ is residually solvable.
\end{proof}

\section{The relator is a basic commutator}
The tools of the Magnus theory were of good use for proving residual
solvability through  gaining  information about the structure of the
two-generator one-relator groups where the relator is a particular
type of basic commutator.

We begin with recalling P.Hall's \cite{hall} definition of the basic
commutators (in terms of the free group $F$ on $\{x_1,...,x_q \}$)
and their linear ordering (in terms of their \emph{weights}).

\begin{definition} Basic Commutators.
\begin{enumerate}
\item  The basic commutators of weight one with their linear order are $x_1 <x_2 < \cdots < x_q$; for their weights we write
$wt(x_i)=1$.
\item    Having defined the basic commutators of weight less than $n$, the basic commutators with weight $n$  are of the form  $c_n=[c_i,c_j]$
where $c_i$  and $c_j$  are all the basic commutators satisfying $wt(c_i)
+wt(c_j)=n$, $c_i>c_j$,  and such that if $c_i=[c_s,c_t]$, then $c_j
\geq c_t$.
\end{enumerate}
\end{definition}

In the following, for positive integers $k$ we will use the notation
$s_1$ = $x$, and $s_{k+1}=[s_k ,y]$.

\begin{theorem}
The group $G= \langle x,y; r=[s_k,y] \rangle$ is free-by-cyclic,
therefore residually solvable.
\end{theorem}
\begin{proof}
Following the Magnus theory, we put $x_i=y^{-i} x y^i$. Using induction on the weight of the commutator and the relationship
\begin{equation}
[s_k,y]=s_k^{-1} (s_k)^y, (k > 0) \n
\end{equation}
we see that the minimum index and maximum index in $r_0$ are $0$ and
$k$, respectively, and both of $x_0$  and $x_k$ only occur once in
$r_0$. By Lemma \ref{mag1} it follows, similarly to previous cases,
that $G$ is free-by-cyclic.
\end{proof}

\section{Open Problems}
\begin{enumerate}
\item Is it algorithmically decidable whether a one-relator group is residually solvable?
\item Are one-relator groups generically residually solvable? In other word, are they in most cases residually solvable?
In \cite{bog}, a conjecture of I.Kapovich is quoted and it states
that many one-relator groups are finitely generated free-by-cyclic.
\item Do there exist residually finite one-relator groups that are not residually
solvable? (This is recasting a question in \cite{ADK} in this
context.)
\item Certain basic commutators make were shown in this paper to produce residually solvable one-relator
groups. Would all basic commutators have that property? If not can
the techniques described here be extended to more one relator groups
where the relator is a basic commutator?
\item Find further examples of non-positive one relator groups that fail to be residually
solvable.
\item Find examples of residually solvable one-relator groups that are not free-by-cyclic.
\end{enumerate}

\end{document}